\documentclass[11pt]{article}
\usepackage{amsmath}
\usepackage{amsthm}
\usepackage{amssymb,amsfonts}
\usepackage{fancyhdr}
\usepackage{hyperref}
\usepackage{latexsym,microtype}
\usepackage{todonotes}
\usepackage{nicefrac}
\usepackage{epsfig,euscript}
\usepackage{stmaryrd}
\usepackage{setspace,mathrsfs}
\usepackage{enumerate}
\usepackage[all]{xypic}
\usepackage{bbm,ifpdf,tikz}
\ifpdf
\usepackage{pdfsync}
\fi

\newtheorem{theorem}{Theorem}[section]

\newtheorem{lemma}[theorem]{Lemma}
\newtheorem{proposition}[theorem]{Proposition}
\newtheorem{corollary}[theorem]{Corollary}

{
\theoremstyle{definition}

\newtheorem{example}[theorem]{Example}

\newtheorem{remark}[theorem]{Remark}
}

\newcommand{\excise}[1]{}

\newcommand{\id}{\operatorname{id}}

\newcommand{\rk}{\operatorname{rk}}

\renewcommand{\and}{\qquad\text{and}\qquad}
\newcommand{\Ind}{\operatorname{Ind}}
\newcommand{\Res}{\operatorname{Res}}

\newcommand{\Z}{\mathbb{Z}}

\newcommand{\N}{\mathbb{N}}

\newcommand{\OS}{OS}

\newcommand{\VRep}{\operatorname{VRep}}

\newcommand{\IhL}{\mathscr{I}^W_{\nicefrac{1}{2}}(L)}
\newcommand{\Ih}{\mathscr{I}^W_{\nicefrac{1}{2}}(P)}

\newcommand{\cC}{\mathcal{C}}

\newcommand{\CWP}{\cC^W\!(P)}
\newcommand{\IWP}{I^W\!(P)}
\newcommand{\IWL}{I^W\!(L)}
\newcommand{\scrIWP}{\mathscr{I}^W\!(P)}
\newcommand{\scrIWL}{\mathscr{I}^W\!(L)}
\newcommand{\bigmid}{\;\Big{|}\;}

\begin{document}
\spacing{1.2}
\noindent{\Large\bf Equivariant incidence algebras and equivariant Kazhdan--Lusztig--Stanley theory}\\

\noindent{\bf Nicholas Proudfoot}\\
Department of Mathematics, University of Oregon,
Eugene, OR 97403\\
njp@uoregon.edu\\

{\small
\begin{quote}
\noindent {\em Abstract.} 
We establish a formalism for working with incidence algebras of posets with symmetries, and we develop
equivariant Kazhdan--Lusztig--Stanley theory within this formalism.  This gives a new way of thinking about
the equivariant Kazhdan--Lusztig polynomial and equivariant $Z$-polynomial of a matroid.
\end{quote} }

\section{Introduction}
The incidence algebra of a locally finite poset was first introduced by Rota, and has proved to be a natural
formalism for studying such notions as M\"obius inversion \cite{Rota-incidence}, generating functions \cite{incidence-generating}, 
and Kazhdan--Lusztig--Stanley polynomials \cite[Section 6]{Stanley-h}.

A special class of Kazhdan--Lusztig--Stanley polynomials that have received a lot of attention recently is that of Kazhdan--Lusztig polynomials
of matroids, where the relevant poset is the lattice of flats \cite{EPW,KLS}.  If a finite group $W$ acts on a matroid $M$ (and therefore on the lattice of flats),
one can define the $W$-equivariant Kazhdan--Lusztig polynomial of $M$ \cite{GPY}.
This is a polynomial whose coefficients are virtual representations of $W$, and has the property that taking dimensions recovers
the ordinary Kazhdan--Lusztig polynomial of $M$.  In the case of the uniform matroid of rank $d$ on $n$ elements, it is actually
much easier to describe the $S_n$-equivariant Kazhdan--Lusztig polynomial, which admits a nice description in terms of partitions of $n$,
than it is to describe the non-equivariant Kazhdan--Lusztig polynomial \cite[Theorem 3.1]{GPY}.

While the definition of Kazhdan--Lusztig--Stanley polynomials is greatly clarified by the language of incidence algebras,
the definition of the equivariant Kazhdan--Lusztig polynomial of a matroid is completely {\em ad hoc} and not nearly as elegant.
The purpose of this note is to define the equivariant incidence algebra of a poset with a finite group of symmetries, and to show
that the basic constructions of Kazhdan--Lusztig--Stanley theory make sense in this more general setting.
In the case of a matroid, we show that this approach recovers the same equivariant Kazhdan--Lusztig polynomials that were defined
in \cite{GPY}.\\

\noindent
{\em Acknowledgments:}
We thank Tom Braden for his feedback on a preliminary draft of this work.

\section{The equivariant incidence algebra}
Fix once and for all a field $k$.
Let $P$ be a locally finite poset equipped with the action of a finite group $W$.
We consider the category $\CWP$ whose objects consist of
\begin{itemize}
\item a $k$-vector space $V$
\item a direct product decomposition $V = \prod_{x\leq y\in P} V_{xy}$, with each $V_{xy}$ finite dimensional
\item an action of $W$ on $V$ compatible with the decomposition.
\end{itemize}
More concretely, for any $\sigma\in W$ and any $x\leq y\in P$, we have a linear map
$\varphi^{\sigma}_{xy}:V_{xy}\to V_{\sigma(x)\sigma(y)}$, 
and we require that $\varphi^{e}_{xy} = \id_{V_{xy}}$ and that
$\varphi^{\sigma'}_{\sigma(x)\sigma(y)}\circ\varphi^{\sigma}_{xy} = \varphi^{\sigma'\sigma}_{xy}$.
Morphisms in $\CWP$ are defined to be linear maps that are compatible with both the decomposition and the action.
This category admits a monoidal structure, with tensor product given by 
$$(U\otimes V)_{xz} := \bigoplus_{x\leq y\leq z}U_{xy}\otimes V_{yz}.$$
Let $\IWP$ be the Grothendieck ring of $\CWP$; we call $\IWP$ the {\bf equivariant incidence algebra} of $P$
with respect to the action of $W$.

\begin{example}
If $W$ is the trivial group, then $\IWP$ is isomorphic to the usual incidence algebra of $P$ with coefficients in $\Z$.
That is, it is isomorphic as an abelian group to a direct product of copies of $\Z$, one for each interval in $P$, and multiplication
is given by convolution.
\end{example}

\begin{remark}\label{base change}
If $W$ acts on $P$ and $\psi:W'\to W$ is a group homomorphism, then $\psi$ induces a functor $F_\psi:\CWP\to\cC^{W'}\!(P)$
and a ring homomorphism $R_\psi:\IWP\to I^{W'}\!(P)$.
\end{remark}

We now give a second, more down to earth description of $\IWP$.
Let $\VRep(W)$ denote the ring of finite dimensional virtual representations of $W$ over the field $k$.
A group homomorphism $\psi:W'\to W$ induces a ring homomorphism $\Lambda_\psi:\VRep(W)\to\VRep(W')$.
For any $x\in P$, let $W_x\subset W$ be the stabilizer of $x$.  We also define 
$W_{xy} := W_x\cap W_y$ and $W_{xyz} := W_x \cap W_y\cap W_z$.
Note that, for any $x,y\in P$ and $\sigma\in W$, conjugation by $\sigma$ gives a group isomorphism 
$$\psi_{xy}^\sigma:W_{xy}\to W_{\sigma(x)\sigma(y)},$$ which induces a ring isomorphism
$$\Lambda_{\psi_{xy}^\sigma}:\VRep(W_{\sigma(x)\sigma(y)})\to \VRep(W_{xy}).$$ 
An element $f\in \IWP$ is uniquely determined by a collection $$\{f_{xy}\mid x\leq y\in P\},$$ where $f_{xy}\in\VRep(W_{xy})$
and for any $\sigma\in W$ and $x\leq y\in P$, $f_{xy} = \Lambda_{\psi_{xy}^\sigma}\left(f_{\sigma(x)\sigma(y)}\right)$.
The unit $\delta\in\IWP$ is characterized by the property that $\delta_{xx}$ is the 1-dimensional trivial representation of $W_x$ for all $x\in P$
and $\delta_{xy} = 0$ for all $x<y\in P$.
The following proposition describes the product structure on $\IWP$ in this representation.

\begin{proposition}\label{multiplication}
For any $f,g\in\IWP$.
$$(fg)_{xz} := \sum_{x\leq y\leq z}\frac{|W_{xyz}|}{|W_{xz}|}\Ind_{W_{xyz}}^{W_{xz}} \left(\left(\Res^{W_{xy}}_{W_{xyz}} f_{xy}\right)\otimes\left( \Res^{W_{yz}}_{W_{xyz}} g_{yz}\right)\right).$$
\end{proposition}

\begin{remark}\label{fractions}
It may be surprising to see the fraction $\frac{|W_{xyz}|}{|W_{xz}|}$ in the statement of Proposition \ref{multiplication},
since $\VRep(W_{xy})$ is not a vector space over the rational numbers.
We could in fact replace the sum over $[x,z]$ with a sum over one representative of each $W_{xz}$-orbit in $[x,z]$ and then eliminate the factor of  
$\frac{|W_{xyz}|}{|W_{xz}|}$.  Including the fraction in the equation allows us to avoid choosing such representatives.
\end{remark}

\begin{remark}
Proposition \ref{multiplication} could be taken as the definition of $\IWP$.  
It is not so easy to prove associativity directly from this definition,
though it can be done with the help of Mackey's restriction formula (see for example \cite[Corollary 32.2]{Bump}).
\end{remark}

\begin{remark}\label{dimension}
Suppose that $\psi:W'\to W$ is a group homomorphism, and for any $x,y\in P$, consider the induced group homomorphism $\psi_{xy}:W'_{xy}\to W_{xy}$.
For any $f\in \IWP$, we have, $R_\psi(f)_{xy} = \Lambda_{\psi_{xy}}\left(f_{xy}\right)$.
In particular, if $W'$ is the trivial group, then $R_\psi(f)_{xy}$ is equal to the dimension of the virtual representation $f_{xy}\in\VRep(W_{xy})$.
\end{remark}

Before proving Proposition \ref{multiplication}, we state the following standard lemma in representation theory.

\begin{lemma}\label{induced reps}
Suppose that $E = \bigoplus_{s\in S} E_s$ is a vector space that decomposes as a direct sum of pieces indexed by a finite set $S$.
Suppose that $G$ acts linearly on $E$ and acts by permutations on $S$ such that, for all $s\in S$ and $\gamma\in G$,
$\gamma\cdot E_s = E_{\gamma\cdot s}$.  For each $x\in S$, let $G_x\subset G$ denote the stabilizer of $s$.
Then there exists an isomorphism
$$E \cong \bigoplus_{s\in S}\frac{|G_s|}{|G|}\Ind_{G_s}^{G} \big(E_s\big)$$
of representations of $G$.\footnote{As in Remark \ref{fractions}, we may eliminate the fraction at the cost of choosing one
representative of each $W$-orbit in $S$.}
\end{lemma}

\begin{proof}[Proof of Proposition \ref{multiplication}.]
By linearity, it is sufficient to prove the proposition in the case where we have objects $U$ and $V$ of $\CWP$
with $f = [U]$ and $g = [V]$.
This means that, for all $x\leq y\leq z\in P$, $f_{xy} = [U_{xy}]\in\VRep(W_{xy})$,
$g_{yz} = [V_{yz}]\in\VRep(W_{yz})$,
and $$(fg)_{xz} = \big[(U\otimes V)_{xz}\big] = \left[\bigoplus_{x\leq y\leq z}U_{xy}\otimes V_{yz}\right]\in \VRep(W_{xz}).$$
The proposition then follows from Lemma \ref{induced reps} by taking $E = (U\otimes V)_{xz}$, $S = [x,z]$, and $G = W_{xz}$.
\end{proof}

Let $R$ be a commutative ring.  Given an element $f\in\IWP\otimes R$ and a pair of elements $x\leq y\in P$, we will write $f_{xy}$ to denote the
corresponding element of $\VRep(W_{xy})\otimes R$.

\begin{proposition}\label{inverses}
An element $f\in\IWP\otimes R$ is (left or right) invertible if and only if $f_{xx}\in\VRep(W_{x})\otimes R$ is invertible for all $x\in P$.
In this case, the left and right inverses are unique and they coincide.
\end{proposition}

\begin{proof}
By Proposition \ref{multiplication}, an element $g$ is a right inverse to $f$ if and only if $g_{xx} = f_{xx}^{-1}$ for all $x\in P$ and 
$$\sum_{x\leq y\leq z}\frac{|W_{xyz}|}{|W_{xz}|}\Ind_{W_{xyz}}^{W_{xz}} \left(\left(\Res^{W_{xy}}_{W_{xyz}} f_{xy}\right)\otimes\left( \Res^{W_{yz}}_{W_{xyz}} g_{yz}\right)\right) = 0$$ for all $x<z\in P$.\footnote{If the ring $R$ has integer torsion, then we rewrite this equation without
the fractions as described in Remark \ref{fractions}.}
The second condition can be rewritten as
$$\left(\Res^{W_x}_{W_{xz}} f_{xx}\right)\otimes g_{xz} = - \sum_{x<y\leq z}\frac{|W_{xyz}|}{|W_{xz}|}\Ind_{W_{xyz}}^{W_{xz}} \left(\left(\Res^{W_{xy}}_{W_{xyz}} f_{xy}\right)\otimes\left( \Res^{W_{yz}}_{W_{xyz}} g_{yz}\right)\right).$$
If $f_{xx}$ is invertible in $\VRep(W_{x})\otimes R$, then $\Res^{W_x}_{W_{xz}} f_{xx}$ is invertible in $\VRep(W_{xz})\otimes R$, and this equation has a unique solution
for $g$.  Thus $f$ has a right inverse if and only if $f_{xx}\in\VRep(W_{x})\otimes R$ is invertible for all $x\in P$.  The argument for left inverses is identical, 
so it remains only to show that left and right inverses coincide.

Let $g$ be right inverse to $f$.  Then $g$ is also left inverse to some function, which we will denote $h$.  We then have
$$f = f\delta = f(gh) = (fg)h = \delta h = h,$$
so $g$ is left inverse to $f$, as well.
\end{proof}

\section{Equivariant Kazhdan--Lusztig--Stanley theory}
In this section we take $R$ to be the ring $\Z[t]$ and for each $f\in\IWP\otimes \Z[t]$ and $x\leq y\in P$, we write $f_{xy}(t)$ for the corresponding component of $f$.
One can regard $f_{xy}(t)$ as a polynomial whose coefficients are virtual representations of $W_{xy}$, or equivalently as a graded virtual representation
of $W_{xy}$.  We assume that $P$ is equipped with a $W$-invariant {\bf weak rank function} in the sense of \cite[Section 2]{Brenti-twisted}.
This is a collection of natural numbers $\{r_{xy}\in\N\mid x\leq y\in P\}$ with the following properties:
\begin{itemize}
\item $r_{xy} > 0$ if $x<y$
\item $r_{xy}+r_{yz}=r_{xz}$ if $x\leq y\leq z$
\item $r_{xy} = r_{\sigma(x)\sigma(y)}$ if $x\leq y$ and $\sigma\in W$.
\end{itemize}
Following the notation of \cite[Section 2.1]{KLS}, we define
$$\scrIWP := \left\{f\in\IWP\otimes \Z[t]\bigmid \text{$\deg f_{xy}(t)\leq r_{xy}$ for all $x\leq y$}\right\}$$
along with $$\Ih := \left\{f\in\IWP\otimes \Z[t]\bigmid \text{$\deg f_{xy}(t)< r_{xy}/2$ for all $x< y$ and $f_{xx}(t) = \delta_{xx}(t)$ for all $x$}\right\}.$$
Note that $\scrIWP$ is a subalgebra of $\IWP$,
and we define an involution $f\mapsto \bar f$ of $\scrIWP$ by putting $\bar f_{xy}(t) := t^{r_{xy}} f_{xy}(t^{-1})$.
An element $\kappa\in\scrIWP$ is called a {\bf \boldmath{$P$}-kernel} if $\kappa_{xx}(t) = \delta_{xx}(t)$ for all $x\in P$ and $\bar\kappa = \kappa^{-1}$.

\begin{theorem}\label{thm:KL}
If $\kappa\in\scrIWP$ is a P-kernel, there exists a unique pair of functions $f,g\in\Ih$ such that
$\bar f = \kappa f$ and $\bar g = g\kappa$.
\end{theorem}

\begin{proof}
We follow the proof in \cite[Theorem 2.2]{KLS}.
We will prove existence and uniqueness of $f$; the proof for $g$ is identical.
Fix elements $x<w\in P$.  Suppose that $f_{yw}(t)$ has been defined for all $x<y\leq w$ and that the equation $\bar f = \kappa f$
holds where defined.
Let $$Q_{xw}(t) := \sum_{x<y\leq w} \frac{|W_{xyw}|}{|W_{xw}|}
\Ind_{W_{xyw}}^{W_{xw}}\left(\left(\Res^{W_{xy}}_{W_{xyw}}\kappa_{xy}(t)\right)\otimes\left( \Res^{W_{yw}}_{W_{xyw}}f_{yw}(t)\right)\right) 
\in \VRep(W_{xw})\otimes \Z[t].$$
The equation $\bar f = \kappa f$ for the interval $[x,w]$ translates to $$\bar f_{xw}(t) - f_{xw}(t) = Q_{xw}(t).$$
It is clear that there is at most one polynomial $f_{xw}(t)$ of degree strictly less than $r_{xw}/2$ satisfying this equation.
The existence of such a polynomial is equivalent to the statement $$t^{r_{xw}}Q_{xw}(t^{-1}) = -Q_{xw}(t).$$
To prove this, we observe that
\begin{eqnarray*}
t^{r_{xw}}Q_{xw}(t^{-1}) &=& t^{r_{xw}}\sum_{x<y\leq w} \frac{|W_{xyw}|}{|W_{xw}|}\Ind_{W_{xyw}}^{W_{xw}}\left(\left(\Res^{W_{xy}}_{W_{xyw}}\kappa_{xy}(t^{-1})\right)\otimes\left( \Res^{W_{yw}}_{W_{xyw}}f_{yw}(t^{-1})\right)\right)\\
&=& \sum_{x<y\leq w} \frac{|W_{xyw}|}{|W_{xw}|}\Ind_{W_{xyw}}^{W_{xw}}\left(\left(\Res^{W_{xy}}_{W_{xyw}}t^{r_{xy}}\kappa_{xy}(t^{-1})\right)\otimes\left( \Res^{W_{yw}}_{W_{xyw}}t^{r_{yw}}f_{yw}(t^{-1})\right)\right)\\
&=& \sum_{x<y\leq w} \frac{|W_{xyw}|}{|W_{xw}|}\Ind_{W_{xyw}}^{W_{xw}}\left(\left(\Res^{W_{xy}}_{W_{xyw}}\bar\kappa_{xy}(t)\right)\otimes\left( \Res^{W_{yw}}_{W_{xyw}}\bar f_{yw}(t)\right)\right)\\
&=& \sum_{x<y\leq w} \frac{|W_{xyw}|}{|W_{xw}|}\Ind_{W_{xyw}}^{W_{xw}}\left(\left(\Res^{W_{xy}}_{W_{xyw}}\bar\kappa_{xy}(t)\right)\otimes\left( \Res^{W_{yw}}_{W_{xyw}}(\kappa f)_{yw}(t)\right)\right).
\end{eqnarray*}
This is formally equal to the expression for $(\bar\kappa(\kappa f))_{xw} - (\kappa f)_{xw}$, which by associativity
is equal to the expression for $$((\bar\kappa \kappa)f)_{xw} - (\kappa f)_{xw} = f_{xw} - (\kappa f)_{xw}.$$
Thus we have
\begin{eqnarray*}
t^{r_{xw}}Q_{xw}(t^{-1}) &=&
- \sum_{x<y\leq w} \frac{|W_{xyw}|}{|W_{xw}|}\Ind_{W_{xyw}}^{W_{xw}}\left(\left(\Res^{W_{xy}}_{W_{xyw}}\kappa_{xy}(t)\right)\otimes\left( \Res^{W_{yw}}_{W_{xyw}}f_{yw}(t)\right)\right)\\
&=& - Q_{xw}(t).
\end{eqnarray*}
Thus there is a unique choice of polynomial $f_{xw}(t)$ consistent with the equation $\bar f = \kappa f$ on the interval $[x,w]$.
\end{proof}

We will refer to the element $f\in \Ih$ from Theorem \ref{thm:KL} is the {\bf right equivariant KLS-function} associated with $\kappa$,
and to $g$ as the {\bf left equivariant KLS-function} associated with $\kappa$.  
For any $x\leq y$, we will refer to the graded virtual representations $f_{xy}(t)$ and $g_{xy}(t)$
as (right or left) {\bf equivariant KLS-polynomials}.
When $W$ is the trivial group, these definitions specialize to the ones in \cite[Section 2]{KLS}.

\begin{example}\label{char}
Let $\zeta\in\scrIWP$ be the element defined by letting $\zeta_{xy}(t)$ be the trivial representation of $W_{xy}$
in degree zero for all $x\leq y$, and let $\chi := \zeta^{-1}\bar\zeta$.  The function $\chi$ is called the {\bf equivariant characteristic
function} of $P$ with respect to the action of $W$.  We have $\chi^{-1} = \bar\zeta^{-1}\zeta = \bar\chi$, so $\chi$ is a $P$-kernel.
Since $\bar\zeta = \zeta\chi$, $\zeta$ is equal to the left KLS-function associated with $\chi$.  However, the right KLS-function $f$
associated with $\chi$ is much more interesting!  See Propositions \ref{char-OS} and \ref{matroid main} for a special case of this construction.
\end{example}

We next introduce the equivariant analogue of the material in \cite[Section 2.3]{KLS}.
If $\kappa$ is a $P$-kernel with right and left KLS-functions $f$ and $g$,
we define $Z := g\kappa f\in\scrIWP$, which we call the {\bf equivariant \boldmath{$Z$}-function} associated with $\kappa$.
For any $x\leq y$, we will refer to the graded virtual representation $Z_{xy}(t)$ as an {\bf equivariant \boldmath{$Z$}-polynomial}.

\begin{proposition}\label{palindromic}
We have $\bar Z = Z$.  
\end{proposition}

\begin{proof}
Since $\bar g = g\kappa$, we have $Z = g\kappa f = \bar g f$.  
Since $\bar f = \kappa f$, we have $Z = g\kappa f = g\bar f$.
Thus $\bar Z = \overline{\bar g f} = \bar{\bar g}\bar f = g\bar f = Z$.
\end{proof}

\begin{remark}\label{restrictions}
Suppose that $\kappa\in\IWP$ is a $P$-kernel and $f,g,Z\in\IWP$ are the associated equivariant KLS-functions and equivariant $Z$-function.
It is immediate from the definitions that, if $\psi:W'\to W$ is a group homomorphism, then $R_\psi(f),R_\psi(g),R_\psi(Z)\in I^{W'}\!(P)$
are the equivariant KLS-functions and equivariant $Z$-function associated with the $P$-kernel $R_\psi(\kappa)\in I^{W'}\!(P)$.
In particular, if we take $W'$ to be the trivial group, then Remark \ref{dimension} tells us that the ordinary KLS-polynomials and $Z$-polynomials
are recovered from the equivariant KLS-polynomials and $Z$-polynomials by sending virtual representations to their dimensions.
\end{remark}

\section{Matroids}\label{sec:matroid}
Let $M$ be a matroid, let $L$ be the lattice of flats of $M$
equipped with the usual weak rank function, and let $W$ be a finite group acting on $L$.
Let $\OS^W_M(t)$ be the Orlik--Solomon algebra of $M$, regarded as a graded representation of $W$.
Following \cite[Section 2]{GPY}, we define
$$H^W_M(t) := t^{\rk M}\OS_M^W(-t^{-1})\in\VRep(W)\otimes\Z[t].$$
If $W$ is trivial, then $H^W_M(t)\in\Z[t]$ is equal to the characteristic polynomial of $M$.
For any $F\leq G\in L$, let $M_{FG}$ be the minor of $M$ with lattice of flats $[F,G]$
obtained by deleting the complement of $G$ and contracting $F$; this matroid inherits an action of the stabilizer group $W_{FG}\subset W$.
Define $H\in\scrIWL$ by putting $H_{FG}(t) = H_{M_{FG}}^{W_{FG}}(t)$ for all $F\leq G$.

\begin{proposition}\label{char-OS}
The function $H$ is equal to the equivariant characteristic function of $L$.
\end{proposition}

\begin{proof}
It is proved in \cite[Lemma 2.5]{GPY} that $\zeta H = \bar \zeta$.  Multiplying on the left by $\zeta^{-1}$, we have $H = \zeta^{-1}\bar\zeta$,
which is the definition of the equivariant characteristic function of $L$.
\end{proof}

\begin{remark}
The proof of \cite[Lemma 2.5]{GPY} is surprisingly difficult.\footnote{The difficult
part appears in the proof of Lemma 2.4, which is then used to prove Lemma 2.5.}
Consequently, Proposition \ref{char-OS} is a deep fact about Orlik--Solomon algebras, not just a formal consequence of the definitions.
\end{remark}

The {\bf equivariant Kazhdan--Lusztig polynomial} $P^W_M(t) \in \VRep(W)\otimes\Z[t]$ was introduced
in \cite[Section 2.2]{GPY}.  
Define $P\in\IhL$ by putting $P_{FG}(t) = P_{M_{FG}}^{W_{FG}}(t)$ for all $F\leq G$.
The defining recursion for $P^W_M(t)$ in \cite[Theorem 2.8]{GPY} translates to the formula $\bar P = H P$, which immediately implies the following proposition.

\begin{proposition}\label{matroid main}
The function $P$ is the right equivariant KLS-function associated with $H$.
\end{proposition}

The {\bf equivariant \boldmath{$Z$}-polynomial} $Z^W_M(t) \in \VRep(W)\otimes\Z[t]$ was introduced
in \cite[Section 6]{PXY}.  Define $Z\in\scrIWL$ by putting $Z_{FG}(t) = Z_{M_{FG}}^{W_{FG}}(t)$ for all $F\leq G$.
The defining recursion for $Z^W_M(t)$ in \cite[Section 6]{PXY} translates to the formula $Z = \bar\zeta P$.

\begin{proposition}\label{matroid Z}
The function $Z$ is the $Z$-function associated with $H$.
\end{proposition}

\begin{proof}
Example \ref{char} tells us that the right KLS-function associated with $H$ is $\zeta$ and Proposition \ref{matroid main} tells us that the
left KLS-function associated with $H$ is $P$, thus the $Z$-function is equal 
$\zeta H P = \bar\zeta P = Z.$
\end{proof}

The following corollary was asserted without proof in \cite[Section 6]{PXY}, and follows immediately from Propositions \ref{palindromic} and \ref{matroid Z}.

\begin{corollary}
The polynomial $Z^W_M(t)$ is palindromic.  That is, $t^{\rk M}Z^W_M(t^{-1}) = Z^W_M(t)$.
\end{corollary}

When $W$ is the trivial group, Gao and Xie define polynomials $Q_M(t)$ and $\hat Q_M(t) = (-1)^{\rk M}Q_M(t)$ with the property that 
$\left(P^{-1}\right)_{FG}(t) = \hat Q_{M_{FG}}(t)$ \cite{GX}.  If $\hat 0$ and $\hat 1$ are the minimal and maximal flats of $M$,
this is equivalent to the statement that $Q_M(t) = (-1)^{\rk M}\left(P^{-1}\right)_{\hat 0 \hat 1}(t)$.
The polynomial $Q_M(t)$ is called the {\bf inverse Kazhdan--Lusztig polynomial of $M$}.\footnote{The reason for bestowing this name on $Q_M(t)$ rather than $\hat Q_M(t)$ is that $Q_M(t)$ has non-negative coefficients.}
Using the machinery of this paper, we may extend their definition to the equivariant setting by defining 
the {\bf equivariant inverse Kazhdan--Lusztig polynomial}
$$Q_M^W(t) := (-1)^{\rk M}\left(P^{-1}\right)_{\hat 0 \hat 1}(t).$$  If we then define $\hat Q\in\IhL$ by putting $\hat Q_{FG}(t) = (-1)^{r_{FG}}Q_{M_{FG}}^{W_{FG}}(t)$ for all $F\leq G$,
we immediately obtain the following proposition.

\begin{proposition}
The functions $P$ and $\hat Q$ are mutual inverses in $\IWL$.
\end{proposition}

\bibliography{./symplectic}

\def\cprime{$'$}
\providecommand{\bysame}{\leavevmode\hbox to3em{\hrulefill}\thinspace}
\providecommand{\MR}{\relax\ifhmode\unskip\space\fi MR }
\providecommand{\MRhref}[2]{%
  \href{http://www.ams.org/mathscinet-getitem?mr=#1}{#2}
}
\providecommand{\href}[2]{#2}
\begin{thebibliography}{EPW16}

\bibitem[Bre99]{Brenti-twisted}
Francesco Brenti, \emph{Twisted incidence algebras and
  {K}azhdan-{L}usztig-{S}tanley functions}, Adv. Math. \textbf{148} (1999),
  no.~1, 44--74.

\bibitem[Bum13]{Bump}
Daniel Bump, \emph{Lie groups}, second ed., Graduate Texts in Mathematics, vol.
  225, Springer, New York, 2013.

\bibitem[DRS72]{incidence-generating}
Peter Doubilet, Gian-Carlo Rota, and Richard Stanley, \emph{On the foundations
  of combinatorial theory. {VI}. {T}he idea of generating function},
  Proceedings of the {S}ixth {B}erkeley {S}ymposium on {M}athematical
  {S}tatistics and {P}robability ({U}niv. {C}alifornia, {B}erkeley, {C}alif.,
  1970/1971), {V}ol. {II}: {P}robability theory, 1972, pp.~267--318.

\bibitem[EPW16]{EPW}
Ben Elias, Nicholas Proudfoot, and Max Wakefield, \emph{The {K}azhdan-{L}usztig
  polynomial of a matroid}, Adv. Math. \textbf{299} (2016), 36--70.

\bibitem[GPY17]{GPY}
Katie Gedeon, Nicholas Proudfoot, and Benjamin Young, \emph{The equivariant
  {K}azhdan--{L}usztig polynomial of a matroid}, J. Combin. Theory Ser. A
  \textbf{150} (2017), 267--294.

\bibitem[GX20]{GX}
Alice L.~L. Gao and Matthew~H.Y. Xie, \emph{The inverse {K}azhdan-{L}usztig
  polynomial of a matroid}, 2020, \textsf{arXiv:2007.15349}.

\bibitem[Pro18]{KLS}
Nicholas Proudfoot, \emph{The algebraic geometry of
  {K}azhdan-{L}usztig-{S}tanley polynomials}, EMS Surv. Math. Sci. \textbf{5}
  (2018), no.~1, 99--127.

\bibitem[PXY18]{PXY}
Nicholas Proudfoot, Yuan Xu, and Ben Young, \emph{The {$Z$}-polynomial of a
  matroid}, Electron. J. Combin. \textbf{25} (2018), no.~1, Paper 1.26, 21.

\bibitem[Rot64]{Rota-incidence}
Gian-Carlo Rota, \emph{On the foundations of combinatorial theory. {I}.
  {T}heory of {M}\"{o}bius functions}, Z. Wahrscheinlichkeitstheorie und Verw.
  Gebiete \textbf{2} (1964), 340--368 (1964).

\bibitem[Sta92]{Stanley-h}
Richard~P. Stanley, \emph{Subdivisions and local {$h$}-vectors}, J. Amer. Math.
  Soc. \textbf{5} (1992), no.~4, 805--851.

\end{thebibliography}
\bibliographystyle{amsalpha}

\end{document}